\documentclass[12 pt]{article}%
\usepackage{amsmath, amsfonts, amsthm, color,latexsym}
\usepackage{amsmath, ulem}
\usepackage{amsfonts}
\usepackage{amssymb}
\usepackage{color, soul}
\usepackage[all]{xy}
\usepackage{graphicx}%
\providecommand{\U}[1]{\protect\rule{.1in}{.1in}}
\allowdisplaybreaks[4]
\newtheorem{theorem}{Theorem}[section]
\newtheorem{proposition}[theorem]{Proposition}
\newtheorem{corollary}[theorem]{Corollary}
\newtheorem{example}[theorem]{Example}

\newtheorem{lemma}[theorem]{Lemma}
\newtheorem{final remark}[theorem]{Final Remark}
\newtheorem{definition}[theorem]{Definition}
\allowdisplaybreaks[4]

\begin{document}

\title{Spaceability of sets of non-injective maps}
\author{Mikaela Aires\thanks{Supported by a CNPq scholarship}~~and Geraldo Botelho\thanks{Supported by FAPEMIG grants RED-00133-21 and APQ-01853-23. \newline 2020 Mathematics Subject Classification: 15A03,  46B87, 47B10, 47L22.\newline Keywords: Spaceability, lineabillity, sequence spaces, function spaces, operator and polynomial ideals, Lipschitz functions. }}
\date{}
\maketitle

\begin{abstract} The purpose of this paper is to extend to spaces of nonlinear operators, and also to more general spaces of linear operators, a recent result on lineability of sets of non-injective linear operators. We prove, for quite general spaces $A$ of (linear or nonlinear) maps from an arbitraty set to a sequence space, that, for every $0 \neq f \in A$, the subset of $A$ of non-injective maps contains an infinite dimensional subspace of $A$ containing $f$. We provide applications to spaces of linear operators between quasi-Banach spaces, to spaces of linear operators belonging to an operator ideal, and, in the nonlinear setting, to spaces of homogeneous polynomials and to spaces of vector-valued Lipschitz functions on metric spaces.
\end{abstract}

\section{Introduction}

Let $A$ be a nonlinear set formed by vectors of an infinite dimensional linear space $E$ satisfying some distinguished property, and let $\alpha$ be a cardinal number not greater than the dimension of $E$. One of the purposes of the fashionable subject of lineability is to decide whether or not there exists an $\alpha$-dimensional subspace $W$ of $E$ such that $W \subset A \cup \{0\}$. If yes, the set $A$ is said to be {\it $\alpha$-lineable}. If $E$ is a topological vector space and $W$ can be chosen to be closed, then $A$ is said to be {\it $\alpha$-spaceable}. For a comprehensive account of this subject, see \cite{Aron}. Lineability and spaceability in spaces of nonlinear operators have already been studied, for example, in spaces of homogeneous polynomials \cite{matos, gamez} and in spaces of Lipschitz functions \cite{aviles, dantas}.
  
  A standard technique in the field consists in fixing $0 \neq x \in A$ and manipulating $x$ conveniently in such a way to construct the subspace $W$. It just so happens that, sometimes, the mother vector $x$ unfortunately does not belong to $W$. Applications of the mother vector technique with happy endings were studied in \cite{Pellegrino}, where $A$ is said to be {\it pointwise $\alpha$-lineable (pointwise $\alpha$-spaceable}) if, for every  $x \in A$, there is a (closed) $\alpha$-dimensional subspace $W$ of $E$ such that $x \in W \subset A \cup \{0\}$. For cardinal numbers $\alpha$ and $\beta$, the quite general notions of $(\alpha,\beta)$-lineability/spaceability were introduced in \cite{Favaro}. Pointwise $\alpha$-lineability/spaceability is closely related to $(1, \alpha)$-lineability/spaceability, but in general these notions are not equivalent (see \cite[Example 2.2]{Pellegrino}). However, for sets of non-injective maps, which happens to be the subject of this paper, these notions are equivalent. Therefore, denoting by $\mathfrak{c}$ the cardinality of the continuum, in our results it is irrelevant if we write that a set is pointwise $\mathfrak{c}$-lineable/spaceable or $(1, \mathfrak{c})$-lineable/spaceable. One word more about terminology: since infinite dimensional Banach spaces have dimension not smaller than $\mathfrak{c}$, we shall write {\it pointwise spaceable} instead of {\it pointwise $\mathfrak{c}$-spaceable}, which, in our context of sets of non-injective maps, is equivalent to $(1, \mathfrak{c})$-spaceable.

Pointwise lineability and $(1, \mathfrak{c})$-lineability of sets of injective/non-injective bounded linear operators between classical spaces were studied in \cite{Diniz, Diniz1, Favaro}. The purpose of this paper is to extend the result we state next to (possibly non-normed) spaces of (possibly nonlinear) operators. 

\begin{theorem} {\rm {\cite[Theorem 3.1]{Favaro}}}\label{6}
For $1 \leq p,q < \infty$, the set of non-injective bounded linear operators from $\ell_p$ to $\ell_q$ is pointwise $\mathfrak{c}$-lineable.
\end{theorem}

We shall generalize the result above in the following directions: (i) The space of bounded linear operators shall be replaced by much more general function spaces, going far enough to include spaces of different types of nonlinear operators. (ii) The domain space $\ell_p$ shall be replaced by an arbitrary set, which is not required even to be a linear space. (iii) The target space $\ell_q$ shall be replaced by much more general sequence spaces, which can be quasi-Banach spaces of vector-valued sequences. (iv) In such a quite general setting, we prove pointwise spaceability instead of pointwise $\mathfrak{c}$-lineability. 

In the linear setting, our main result generalizes Theorem \ref{6} to the range $0 < p,q < \infty$ and to the much smaller sets of non-injective compact operators and non-injective nuclear operators. In the nonlinear setting, we provide applications of our main result to sets of non-injective homogeneous polynomials, to sets of  non-injective bounded Lipschitz functions and to sets of non-injective Lipschitz functions that fix a distinguished point. 

By $B_E$ we denote the closed unit ball of the Banach space $E$, and by $E^*$ we denote its topological dual. For Banach space theory we refer to \cite{meg}

\section{Spreading sequence spaces and stable function spaces}

In this section we introduce the spaces we shall use to generalize Theorem \ref{6}. The abstract definitions are followed by a number of concrete examples.  For the theory of quasi-Banach spaces and $p$-Banach spaces, $0 < p \leq 1$, we refer to \cite{Kalton1, Kalton, Stiles}. Unless explicitly stated otherwise, all linear spaces are over $\mathbb{K} = \mathbb{R}$ or $\mathbb{C}$. Henceforth, all linear spaces and subspaces, including Banach and quasi-Banach spaces, are supposed to be non-null, that is, different from $\{0\}$.

\begin{definition}\rm Let $0 < q \leq 1$ and let $W $ be a linear space. A {\it $W$-spreading $q$-space} is a linear subspace $V $ of the space $W^\mathbb{N}$ of $W$-valued sequences with the coordinatewise operations, endowed with a complete $q$-norm $\|\cdot\|_V$ such that the following condition holds:

$\bullet$ If $(a_j)_{j=1}^\infty \in V$ and $\mathbb{N}_0=\{j_1<j_2< j_3< \cdots \}$ is an increasing infinite subset of $\mathbb{N}$, then the sequence $(b_j)_{j=1}^\infty$ given by
$$b_k =\left\{
\begin{array}{cl}
a_i, & \textrm{if} ~ k=j_i,\\
0, & \textrm{otherwise},
\end{array}
\right.$$
belongs to $V$ and  $ \|(b_j)_{j=1}^\infty\|_V \leq \|(a_j)_{j=1}^\infty\|_V.$
In this case we say that the sequence $(b_j)_{j=1}^\infty$ is the {\it spreading} of $(a_j)_{j=1}^\infty$ with respect $\mathbb{N}_0$, and we write $(b_j)_{j=1}^\infty = {\rm Sp}((a_j)_{j=1}^\infty; \mathbb{N}_0)$.

In the normed case $q = 1$ we write {\it $W$-spreading space} instead of $W$-spreading $1$-space.
\end{definition}

\begin{example}\label{6ymz}\rm Most usual scalar-valued or vector-valued sequence spaces are spreading spaces. Let $E$ be a Banach space.\\
 (i) The following are $E$-spreading spaces with respect to the supremum norm $\|\cdot\|_\infty$: the space $c_0(E)$ of norm null sequences, the space $\ell_\infty(E)$ of bounded sequences and the space $c_0^w(E)$ of weakly null sequences. \\
(ii) If $1 \leq p < \infty$, then the space $\ell_p(E)$ of absolutely $p$-summable sequences is an $E$-spreading space with respect to its usual norm $\|\cdot\|_p$. If $0 < p < 1$, then it is a $E$-spreading $p$-space.\\
(iii) If $1 \leq p < \infty$, then the space $\ell_p^w(E)$ of weakly $p$-summable sequences endowed with its usual norm $\|(x_n)_{n=1}^\infty\|_{w,p} = \sup\limits_{x^* \in B_{E^*}}\|(x^*(x_n))_{n=1}^\infty \|_p$,  and its closed subspace
$$\ell_p^u(E)= \left\{(x_n)_{n=1}^\infty \in \ell_p^w(E) : \lim_{n \to \infty}\|(x_j)_{j=n}^\infty\|_{w,p} = 0\right\},$$
of unconditionally $p$-summable sequences (see \cite[8.2]{df}), are $E$-spreading spaces. If $0 < p < 1$, then they are $E$-spreading $p$-spaces.

As a counterexample, it is clear that the Banach space $c$ of scalar-valued convergent sequences with the supremum norm is not a $\mathbb{K}$-spreading space.
\end{example}

We shall use a couple of times that every Banach space is a $q$-Banach space for any $0 < q < \infty$. Actually, from \cite[Proposition 5.5.2]{garling} it follows that, if $0 < q < p \leq 1$, then every $p$-norm is a $q$-norm, hence every $p$-Banach space is a $q$-Banach space.

\begin{definition}\label{2xum}\rm Let  $\Omega$ be an arbitrary nonempty set, let $F$ be a $q$-Banach space, $0 < q \leq 1$, and, as usual, let $F^\Omega$ denote the linear space of maps from $\Omega$ to $F$ with the pointwise operations. A linear subspace $A \neq \{0\}$ of $F^\Omega$, endowed with a complete $q$-norm $\|\cdot\|_A$, is said to be an {\it $(\Omega, F)$-stable function space} if the following conditions hold:\\
{\rm (i)} The (metrizable) topology on $A$ generated by the $q$-norm $\|\cdot\|_A$ contains the topology of pointwise convergence.\\
{\rm (ii)} If $g \in A$ and $u \colon F \longrightarrow F$ is a bounded linear operator, then $u \circ g \in A$ and $$\|u \circ g\|_A \leq \|u\|\cdot\|g\|_A.$$
\end{definition}

\begin{example}\label{ikm4}\rm (a) Let $0 < p,q \leq 1$, let $E$ be a $p$-Banach space and let $F$ be a $q$-Banach space. It is immediate that the space $\mathcal{L}(E;F)$ of bounded linear operators from $E$ to $F$, endowed with the $q$-norm given by $\|T\|= \sup\limits_{\|x\|\leq 1} \|T(x)\|$, is an $(E,F)$-stable function space.\\
(b) For the theory of quasi-Banach operator ideals, we refer the reader to \cite[Section 9]{fg} and \cite[Part 2]{pietsch}. Let $({\cal I}, \|\cdot\|_{\cal I})$ be a $q$-Banach operator ideal, where $0 < q \leq 1$. Since Banach spaces are $q$-Banach spaces, ${\cal I}(E;F)$ is an $(E,F)$-stable function space for all Banach spaces $E$ and $F$. Denoting by $\|\cdot\|$ the usual norm on ${\cal L}(E;F)$, the inequality $\|\cdot\| \leq \|\cdot\|_{\cal I}$ \cite[Proposition 6.1.4]{pietsch} gives condition (i) of the definition. Condition (ii) follows from the ideal inequality of $\|\cdot\|_{\cal I}$ (just take $m=1$ in (\ref{lm9q})).
\end{example}

 For the reader's convenience, we recall now the notion of quasi-Banach ideal of homogeneous polynomials (see \cite{ewerton, fg}). By ${\cal P}(^mE;F)$ we denote the Banach space of continuous $m$-homogeneous polynomials $P$ from the Banach space $E$ to the Banach space $F$ with the norm $\|P\| = \sup\limits_{x \in B_E}\|P(x)\|$. Given $x^* \in E^*$ and $b \in F$, by $(x^*)^m \otimes b$ we denote the $m$-homogeneous polynomial defined by $[(x^*)^m \otimes b](x) = x^*(x)^mb$.   For the basic theory of spaces of homogeneous polynomials between Banach spaces, see \cite{dineen, mujica}. For $0 < q \leq 1$, a {\it $q$-Banach ideal of homogeneous polynomials}, or simply, a {\it $q$-Banach polynomial ideal} is a subclass $\mathcal{Q}$ of the class ${\cal P}$ of all continuous homogeneous polynomials between Banach spaces, endowed with a function $\|\cdot\|_{\cal Q} \colon {\cal Q} \longrightarrow \mathbb{R}$, such that, for every $m \in \mathbb{N}$ and all Banach spaces $E$ and $F$, the component
$$\mathcal{Q}(^mE;F):=\mathcal{P}(^mE;F)\cap \mathcal{Q}$$
satisfies the following conditions:\\ 
$\bullet$ $\mathcal{Q}(^mE;F)$ is a linear subspace of $\mathcal{P}(^mE;F)$ containing the polynomials of the type $(x^*)^m\otimes b$, $x^*\in E^*$ and $b \in F$.\\
%
$\bullet$ The restriction of $\|\cdot\|_\mathcal{Q}$ to $\mathcal{Q}(^mE;F)$ is  a $q$-norm and
$$\|I_m \colon \mathbb{K} \longrightarrow \mathbb{K}\,,\, I_m(\lambda) = \lambda^m\|_{\cal Q} = 1. $$
$\bullet$ If $u \in \mathcal{L}(F_0;F)$, $P \in \mathcal{Q}(^mE_0;F_0)$ and $v \in \mathcal{L}(E;E_0)$, then $u \circ P \circ v \in \mathcal{Q}(^mE;F)$ and
\begin{equation} \label{lm9q}\|u \circ P \circ v\|_\mathcal{Q} \leq \|u\|\cdot \|P\|_\mathcal{Q} \cdot \|v\|^m.\end{equation}

Plenty of examples, concerning polynomials that are approximable, compact, weakly compact, of absolutely summing-type and of nuclear-type, can be found in \cite{ewerton, klausminimal, fg}. Techniques to generate polynomial ideals from a given operator ideal, which also provide a number of examples, can be found, e.g., in \cite{aronrueda, note, prims}.

\begin{example}\label{is9t}\rm Let $({\cal Q}, \|\cdot\|_{\cal Q})$ be a $q$-Banach polynomial ideal, $0 < q \leq 1$. For all $m \in \mathbb{N}$ and Banach spaces $E,F$, ${\cal Q}(^mE,F)$ is an $(E,F)$-stable function space (we are using again that Banach spaces are $q$-Banach spaces). Condition (i) of the definition follows from the inequality $\|\cdot\| \leq \|\cdot\|_{\cal Q}$, where $\|\cdot\|$ is the usual norm on ${\cal P}(^mE;F)$ (see \cite[Remark 2.2]{ewerton}). Condition (ii) follows from (\ref{lm9q}). 
\end{example}

All metric spaces in this paper are supposed to have at least two points.

\begin{example}\label{l0mz}\rm Let $(M,d)$ be a metric space and let $E$ be a Banach space. A map $f \colon M \longrightarrow E$ is a {\it Lipschitz function} if
$$\|f\|_{d}: = \sup\left\{\frac{\|f(x) - f(y)\|}{d(x,y)} : x,y \in M, x\neq y \right\}< \infty. $$
According to \cite{johnson}, the set ${\rm Lip}(M,E)$ of bounded Lipschitz functions from $M$ to $E$ is a Banach space with the norm $\|f\|_{\rm Lip} := \max\{\|f\|_{d}, \|f\|_{\infty}\}$. The inequality $\|f\|_{\infty} \leq \|f\|_{\rm Lip}$ gives condition (i) of Definition \ref{2xum}. Given $f \in {\rm Lip}(M,E)$ and $u \in {\cal L}(E,E)$, it is easy to check that $\|u \circ f  \|_{d}\leq \|u\|\!\cdot \!\|f\|_{d}$ and  $\|u \circ f  \|_\infty\leq \|u\|\!\cdot \!\|f\|_\infty$. Hence $u \circ f \in {\rm Lip}(M,E)$ and $\|u \circ f  \|_{\rm Lip}\leq \|u\|\!\cdot \!\|f\|_{\rm Lip}$, which gives condition (ii). Therefore, ${\rm Lip}(M,E)$ is an $(M,E)$-stable function space.
\end{example}

\begin{example}\label{km6u}\rm By a pointed metric space we mean a metric space $M$ in which a point $0 \in M$ has been distinguished. The space ${\rm Lip}_0(M,E)$ of all (possibly unbounded) Lipschitz functions $f$ from a pointed metric space $M$ to a Banach space $E$ such that $f(0) = 0$ is a Banach space endowed with the norm $\|\cdot\|_d$ (see, e.g., \cite{brudnyi}). For all $x \in M$ and $f \in {\rm Lip}_0(M,E)$, $\|f(x)\| \leq \|f\|_d\cdot d(x,0)$, which gives condition (i) of Definition \ref{2xum}. Condition (ii) follows as in the example above. Therefore, ${\rm Lip}_0(M,E)$ is an $(M,E)$-stable function space.
\end{example}


\section{Main result}
A couple of preparatory results are needed to prove our main result.

\begin{lemma} Let $W $ be a linear space and $0 < q \leq 1$. Any $W$-spreading $q$-space is at least $\mathfrak{c}$-dimensional.
\end{lemma}

\begin{proof} Let $V $ be a $W$-spreading $q$-space. Using Baire's category theorem, it is not difficult to check that every infinite dimensional $q$-Banach space is at least $\mathfrak{c}$-dimensional (see \cite[Proposition III.5]{Aron}, \cite[Theorem I-1]{mackey}). Therefore, it is enough to prove that $V$ is infinite dimensional. To do so, choose $ v = (v_j)_{j=1}^\infty \in V \backslash \{0\}$ and split $\mathbb{N}$ into countably many pairwise disjoint subsets   $(\mathbb{N}_k)_{k=1}^\infty$. 
For each $k \in \mathbb{N}$, we write $\mathbb{N}_k:=\{n_1^{(k)}< n_2^{(k)}< \cdots\}.$ Given $k_0 \in \mathbb{N}$, define $\pi_{k_0} \colon W^{\mathbb{N}} \longrightarrow W$ by $\pi_{k_0}((x_n)_{n=1}^\infty) = x_{k_0}$. Since
     $v \neq 0$, there is $j_0 \in \mathbb{N}$ so that  $v_{j_0}=\pi _{j_0}(v) \neq 0$. Using  that $V$ a $W$-spreading space, for every $k \in \mathbb{N}$, the sequence $w_k=(w_j^{(k)})_{j=1}^\infty$ given by
    $$w_j^{(k)}=\left\{
\begin{array}{cll}
v_i, & \textrm{if} & j=n_i^{(k)} \in \mathbb{N}_k  ,\\
0, &\textrm{if} & j \notin \mathbb{N}_k
\end{array}
\right.$$
belongs to $V$. We claim that the set $\{w_k : k \in \mathbb{N} \}$ is linearly independent. To prove the claim, let $\alpha_1, \dots , \alpha _m$ be given scalars, $m \in \mathbb{N}$, and suppose that
 \begin{equation}\label{8}
    \alpha _1 w_1 +\alpha _2 w_2 + \cdots + \alpha _m w_m= 0 \in V.
    \end{equation}
In particular,
    $$\alpha _1 \pi _{n_{j_0}^{(1)}}(w_1)+\alpha _2 \pi _{n_{j_0}^{(1)}}(w_2) + \cdots + \alpha _m \pi _{n_{j_0}^{(1)}}(w_m)=0 \in W.$$
Note that $\pi _{n_{j_0}^{(1)}}(w_1)=v_{j_0} \neq 0$ and, since $n_{j_0}^{(1)} \notin \bigcup\limits_{k=2}^\infty \mathbb{N}_k$, $\pi _{n_{j_0}^{(1)}}(w_j)=0$ for every $2\leq j \leq m$. 
Thus, $\alpha_1 v_{j_0}=0$, from which we get $\alpha_1 =0$. Therefore, (\ref{8}) collapses to
$$\alpha _2 w_2 + \cdots + \alpha _m w_m=0.$$
In particular,
\begin{equation} \label{lomw}\alpha _2 \pi _{n_{j_0}^{(2)}}(w_2) + \cdots + \alpha _m \pi _{n_{j_0}^{(2)}}(w_m)=0 \in W.\end{equation}
Using that $n_{j_0}^{(2)} \in \mathbb{N}_2$ and $\mathbb{N}_i \cap \mathbb{N}_j = \emptyset $ for all $i\neq j$, we have $\pi _{n_{j_0}^{(2)}}(w_2)=v_{j_0} \neq 0$ and $\pi _{n_{j_0}^{(2)}}(w_j)=0$ for every $3\leq j \leq m$. From (\ref{lomw}) we get $\alpha _2 v_{j_0}=0$, hence $\alpha _2 =0.$ Repeating the procedure finitely many times we conclude that $\alpha _1 =\cdots = \alpha _m =0$, which shows that $V$ is infinite dimensional and completes the proof. 
\end{proof}

We skip the easy proof of the following lemma.

\begin{lemma} If $V$ is an infinite dimensional linear space and $\Omega$ is a nonempty set, then $ V ^{\Omega}$ is infinite dimensional.
\end{lemma}

%
%
%
%

\begin{theorem} \label{7} Let $W$ be a linear space, let $V$ be a $W$-spreading $q$-space, $0 < q \leq 1$, let $\Omega$ be a nonempty set and let $A$ be a $(\Omega,V)$-stable function space. Then the subset of $A$ of non-injective maps is either 
     $\{0\}$ or  pointwise spaceable.
\end{theorem}

\begin{proof} Suppose that $\mathcal{N}: = \{f \in A: f \textrm{ is non-injective}\}\neq \{0\}.$ Pick $0 \neq f \in \mathcal{N} $ and let  ${\rm span}\{f\} $ be the 1-dimensional subspace generated by $f$. It is clear that ${\rm span}\{f\}  \subseteq {\cal N}$ (this is exactly the reason why pointwise spaceability and $(1, \mathfrak{c})$-spaceability are equivalent for $\cal N$). 
Being $f$ non-injective, there are 
 \begin{equation}\label{1}
 w_0, t_0 \in \Omega,~ w_0\neq t_0, \mbox{ such that }   f(w_0)=f(t_0).
 \end{equation}
For every $w \in \Omega$, we denote by $(f(w))_j$ the $j$-th coordinate of $f(w)$, that is, $f(w) = ((f(w))_j)_{j=1}^\infty$. As $f\neq 0$, there exists $z \in \Omega$ such that $f(z)\neq 0$, hence there is $j_0 \in \mathbb{N}$ such that $(f(z))_{j_0}\neq 0$.
Let $(\mathbb{N}_k)_{k=1}^\infty$ be a sequence of pairwise disjoint infinite subsets of $\mathbb{N}$ not containing $j_0$, that is, $j_0 \notin \bigcup\limits _{k=1}^\infty \mathbb{N}_k$. For each $k \in \mathbb{N}$ we write $\mathbb{N}_k=\{n_1^{(k)}<n_2^{(k)}<\cdots \}$ and consider the map
 $$u_k \colon V \longrightarrow V~,~(u_k(x))_j=\left\{
\begin{array}{cll}
x_i, &\textrm{if} & j=n_i^{(k)} \in \mathbb{N}_k,\\
0, & \textrm{if} &  j \notin \mathbb{N}_k,
\end{array}
\right.$$
where $x = (x_n)_{n=1}^\infty$. The map $u_k$ is well defined, in the sense that it is $V$-valued, because $V$ is a $W$-spreading $q$-space.  
Given $x=(x_n)_{n=1}^\infty, y=(y_n)_{n=1}^\infty \in V$ and $\lambda \in \mathbb{K}$, we have $x+\lambda y = (x_n+\lambda y_n)_{n=1}^\infty$. If $j=n_i^{(k)} \in \mathbb{N}_k$, then
$$(u_k(x+\lambda y))_j= x_i+\lambda y_i =(u_k(x))_j+\lambda (u_k(y))_j.$$
And if $j \notin \bigcup\limits _{k=1}^\infty \mathbb{N}_k$, then the equality above holds with $0 = 0$. This proves that $u_k$ is linear.
For each $k \in \mathbb{N}$, define $f_k:=u_k\circ f \colon  \Omega \longrightarrow V$. We claim that $f_k \in A$ for every $k \in \mathbb{N}$. Since $A$ is a $(\Omega, V)$-stable fucntion space, it is enough to show that the linear operator $u_k$ is bounded. This is true because, since $u_k(x)= {\rm Sp}(x,\mathbb{N}_k)$ for every $x \in V$, we have
$$\|u_k(x)\|_V=\|{\rm Sp}(x,\mathbb{N}_k)\|_V\leq \|x\|_V,$$
which gives that $u_k$ is bounded with $\|u_k\|\leq 1$. 
Therefore, $f_k=u_k \circ f \in A$  and
\begin{equation}\label{9}
\|f_k\|_A=\|u_k \circ f\|_A \leq \|u_k\|\cdot\|f\|_A\leq \|f\|_A
\end{equation}
for every $k \in \mathbb{N}$. Note that, for every $w \in \Omega$: (i)  If $j\notin \mathbb{N}_k$, then  $(f_k(w))_j=(u_k(f(w)))_j=0.$ (ii) If $ j= n_m^{(k)}$ for some $ m \in \mathbb{N}$, then $(f_k(w))_j=(u_k(f(w)))_j=(f(w))_m .$ In other words, $f_k(w)$ can be written in the form
\begin{equation}\label{ghn4}(f_k(w))_j=\left\{
\begin{array}{ccl}
(f(w))_m, &\textrm{if} &  j=n_m^{(k)} \in \mathbb{N}_k,\\
0, & \textrm{if}&  j \notin \mathbb{N}_k.
\end{array}
\right.
\end{equation}
From $f(w_0)=f(t_0)$ if follows now immediately that
\begin{equation}\label{kv4f}{f_k(w_0)=f_k(t_0) \mbox{~for every~} k \in \mathbb{N}.}\end{equation} Therefore $f_k$ is non-injective, that is, $f_k \in \mathcal{N}$ para every $k \in \mathbb{N}.$ Our next purpose is to show that the set $\{f, f_k: k \in \mathbb{N}\}$ is linearly independent. To do so, let $k \in \mathbb{N}$ and let $a,a_1, \dots, a_k \in \mathbb{K}$ be such that
\begin{equation} \label{2}
af+a_1f_1+\cdots +a_kf_k=0.
\end{equation}
Evaluating at the element $z \in \Omega$ for which $(f(z))_{j_0}\neq 0$, we have
$$af(z)+a_1f_1(z)+\cdots + a_kf_k(z)=0.$$
In particular,
\begin{equation} \label{3}
a(f(z))_{j_0}+a_1(f_1(z))_{j_0}+\cdots +a_k(f_k(z))_{j_0}=0.
\end{equation}
Since $j_0 \notin \bigcup\limits_{k=1}^\infty \mathbb{N}_k $, we have $(f_1(z))_{j_0}=\cdots =(f_k(z))_{j_0}=0$, hence $(\ref{3})$ gives $a(f(z))_{j_0}=0$, from which we conclude that $a=0$. Therefore, (\ref{2}) collapses to $a_1f_1+\cdots +a_kf_k=0,$ which implies, in particular, that
\begin{equation} a_1(f_1(z))_{n_{j_0}^{(1)}}+\cdots +a_k(f_k(z))_{n_{j_0}^{(1)}}=0. \label{uy2c} \end{equation}
 Since $\mathbb{N}_1=\{n_1^{(1)}, n_2^{(1)}, \dots n_{j_0}^{(1)}, \dots\}$ and, for every $w \in \Omega$, $(f_1(w))_j=(f(w))_i$ whenever $j=n_i^{(1)} \in \mathbb{N}_1$, we have
$$(f_1(z))_{n_{j_0}^{(1)}}=(f(z))_{j_0}\neq 0.$$
Noting that $(f_k(z))_{n_{j_0}^{(1)}}=0$ for every $k>1$ because $n_{j_0}^{(1)} \notin \bigcup\limits_{k=2}^\infty \mathbb{N}_k$, (\ref{uy2c}) collapses to $a_1(f_1(z))_{n_{j_0}^{(1)}}=0$, which implies that $a_1=0$. Hence, $a_2f_2 + \ldots + a_k f_k = 0$. Repeating the latter argument, we conclude that $a = a_1=\dots = a_k=0.$ Thus far we have proved that $\{f, f_k: k \in \mathbb{N}\}$ is an infinite linearly independent subset of $A$ contained in $\cal N$. In particular, $A$ is infinite dimensional. Let us check that the map
$$\psi \colon \ell_q \longrightarrow A~,~\psi ((a_k)_{k=1}^\infty)=a_1f+\sum _{j=2}^\infty a_jf_{j-1},$$
is well defined. For every $n \in \mathbb{N}$,
\begin{align*}
\|a_1f\|_A^q+ \displaystyle \sum_{j=2}^n \|a_jf_{j-1}\|_A^q & = |a_1|^q\cdot \|f\|_A^q+\displaystyle \sum_{j=2}^n |a_j|^q \cdot\|f_{j-1}\|_A^q \\ &\stackrel{(\ref{9})}{\leq}  |a_1|^q \cdot \|f\|_A^q+\displaystyle \sum _{j=2}^n |a_j|^q\cdot \|f\|^q_A \\
&= \|f\|_A^q \cdot \displaystyle \sum_{j=1}^n |a_j|^q\leq \|f\|_A ^q \cdot \displaystyle \sum _{j=1}^\infty |a_j|^q < \infty.
\end{align*}
Letting $n \rightarrow \infty$ we get $\|a_1f\|_A^q+ \sum\limits_{j=2}^\infty \|a_jf_{j-1}\|_A^q < \infty$, that is, the series $a_1f+\sum\limits_{j=2}^\infty a_jf_{j-1}$ is $q$-absolutely convergent in the $q$-Banach space $A$. From \cite[Lemma 3.2.5]{Aron} it follows that the series converges in $A$, proving that $\psi$ is well defined. The linearity of $\psi$ follows easily. Let us prove now that $\psi$ is injective. 
Indeed, if $\psi ((a_k)_{k=1}^\infty)=0$, then
\begin{equation}\label{4}
a_1f+\sum _{j=2}^\infty a_jf_{j-1}=0.
\end{equation}
Applying at $z$ for which $(f(z))_{j_0}\neq 0$, we get
$$a_1(f(z))_{j_0}+ \sum _{j=2}^\infty a_j(f_{j-1}(z))_{j_0}=0. $$
Since $j_0 \notin \bigcup\limits_{k=1}^\infty \mathbb{N}_k$, we have $(f_{j-1}(z))_{j_0}=0$ for every $j \geq 2$. Hence,
$$\sum _{j=2}^\infty a_j(f_{j-1}(z))_{j_0}=0,$$
from which it follows that $a_1(f(z))_{j_0}=0$, that is, $a_1=0$. So, equality (\ref{4}) collapses to $\sum\limits _{j=2}^\infty a_jf_{j-1}=0.$ The same reasoning we made in (\ref{2}) allows us to conclude that  $a_j=0$ for every $j.$ This proves that $\psi$ is injective. We have that $\psi$ is a linear injective operator on the $\mathfrak{c}$-dimensional space $\ell_q$, therefore $\psi (\ell _q)$ is a $\mathfrak{c}$-dimensional subspace of $A$. Since $f = \psi((1,0,0, \ldots))$, we have $f \in \psi(\ell_q)$. For every $(a_k)_{k=1}^\infty \in \ell _q$ with $a_k \neq 0$ for some $k \in \mathbb{N}$ we have 
\begin{align*}\psi ((a_k)_{k=1}^\infty)(w_0) & =  a_1 f(w_0)+ a_2 f_1(w_0)+a_3 f_2(w_0)+\cdots \\
&\stackrel{\rm(\ref{1}),(\ref{kv4f})}{=}  a_1 f(t_0)+ a_2 f_1(t_0)+a_3 f_2(t_0)+\cdots \\
&= \psi ((a_k)_{k=1}^\infty)(t_0),
\end{align*}
showing that $\psi ((a_k)_{k=1}^\infty)$ is non-injective, that is, $\psi (\ell _q)$ is a $\mathfrak{c}$-dimensional subspace of $A$ contained in $\mathcal{N}$ and containing $f$. 
All that is left to prove is that $\overline{\psi (\ell _q)} \subset \mathcal{N}.$ Given $g \in \overline{\psi (\ell _q)}$, as the closure is taken in the metrizable topology of $A$, we can take a sequence $(g_k)_k$ in $\psi (\ell _q)$ such that $g_k \stackrel{k \to \infty}{\longrightarrow} g$ with respect to the $q$-norm of $A$. For every $k$ there is $(a_j^{(k)})_{j=1}^\infty \in \ell _q$ such that
$$
 g_k= \psi\left((a_j^{(k)})_{j=1}^\infty \right) =  a_1^{(k)}f+\sum_{j=2}^\infty a_j^{(k)}f_{j-1} .$$
For every $k \in \mathbb{N}$,
 \begin{equation}
g_k(w_0)  =  a_1^{(k)}f(w_0)+\sum_{j=2}^\infty a_j^{(k)}f_{j-1}(w_0) \stackrel{\rm (\ref{1}),(\ref{kv4f})}{=}  a_1^{(k)}f(t_0)+\sum_{j=2}^\infty a_j^{(k)}f_{j-1}(t_0)= g_k(t_0).\label{lay9}
\end{equation}
Since $A$ is a $(\Omega,V)$-stable space, its topology contains the topology of pointwise convergence, hence the convergence $g_k \longrightarrow g$ in $A$ implies that $g_k(w) \longrightarrow g(w)$ for every $w \in \Omega$. In particular,
$$g(w_0) = \lim_{k \to \infty} g_k(w_0)\stackrel{\rm (\ref{lay9})}{=}  \lim_{k \to \infty} g_k(t_0) = g(t_0).$$
This shows that $g$ is non-injective, that is, $g \in {\cal N}$. Thus, $\overline{\psi (\ell _q)}$ is a closed infinite dimensional subspace of $A$ contained in $\mathcal{N}$ containing $f$. The proof is complete. 
\end{proof}

\section{Applications}
In this section we use the examples provided in Section 2 to give applications of the main theorem which go far beyond the linear and normed scope of Theorem \ref{6}.

Although keeping the linear environment of Theorem \ref{6}, the first application  encompasses much more general domain and target spaces and, moreover, gives pointwise spaceability.

\begin{corollary}\label{pl7y} Let $0 < p,q \leq 1$, let $E$ be a $p$-Banach space such that  $\dim E >1$ and $E^* \neq \{0\}$, let $W$ be a linear space and let $V$ be $W$-spreading $q$-space. Then the set of non-injective bounded linear operators from $E$ to $V$ is pointwise spaceable in ${\cal L}(E,V)$. 
\end{corollary}

\begin{proof} By Example \ref{ikm4}(a) we know that ${\cal L}(E;V)$ is a $(E,V)$-stable function space. Taking $0 \neq x^* \in E^*$ and $0 \neq b \in V$, using that $\dim E >1$ it is not difficult to check that $x^* \otimes b$ is a non-null non-injective bounded linear operator. The result follows from Theorem \ref{7}.
\end{proof}

Given a Banach space $F$ and $0 < q < \infty$, we consider the set $V_{q,F} = \{c_0(F), c_0^w(F), \ell_\infty(F),$ $ \ell_q(F), \ell_q^w(F), \ell_q^u(F)\}$.

\begin{example}\rm Let $0 < p\leq 1$, let $E$ be a $p$-Banach space such that  $\dim E >1$ and $E^* \neq \{0\}$, and let $F$ be a Banach space. For any $0 < q < \infty$ and every $V \in V_{q,F}$, from Example \ref{6ymz} and the corollary above it follows that the set of non-injective bounded linear operators from $E$ to $V$ is pointwise spaceable.
\end{example}

The next application, although considering operators between Banach spaces, goes a bit further by assuring pointwise spaceability of sets much smaller than the set of non-injective bounded linear operators.

\begin{corollary}  \label{3.10} Let $E$ be a Banach space with ${\rm dim}E > 1$, let $W$ be a linear space, let $V$ be a $W$-spreading space and let $(\mathcal{I}, \|\cdot\|_{\cal I})$ be a $q$-Banach operator ideal, $0 < q \leq 1$. Then the subset of ${\cal I}(E;V)$ of non-injective bounded  linear operators is pointwise spaceable.
\end{corollary}

\begin{proof} By Example \ref{ikm4}(b) we know that $\mathcal{I}(E;V)$ is a $(E,V)$-stable function space. The non-null non-injective operator $x^* \otimes b$ of the proof of Corollary \ref{pl7y} has finite rank, hence it belongs to ${\cal I}(E;V)$. The result follows from Theorem \ref{7}.
\end{proof}

\begin{example}\rm Let $E$ and $F$ be Banach spaces and let $(\mathcal{I}, \|\cdot\|_{\cal I})$ be a $q$-Banach operator ideal, $0 < q \leq 1$. For any $1 \leq p \leq \infty$ and every $V \in V_{p,F}$, from Example  \ref{6ymz} and the corollary above it follows that the subset of ${\cal I}(E;V)$ of non-injective operators is pointwise spaceable. In particular, the sets of non-injective compact operators (with the usual operator norm) and of non-injective nuclear operators (with the nuclear norm) from $E$ to $V$, which are, in general, much smaller than the set of non-injective bounded operators, are pointwise spaceable.
\end{example}

Now we proceed to give applications in spaces of nonlinear operators. In the next result we consider only $m$-homogeneous polynomials with $m \geq 2$ odd in real Banach spaces. The reason is that, if $m$ is even or $\mathbb{K} = \mathbb{C}$, then $m$-homogeneous polynomials are never injective, therefore this case is not of interest.

\begin{corollary} Let $\mathbb{K} = \mathbb{R}$, let $E$ be a Banach space, let $W$  be a linear space, let be $V$ be a $W$-spreading space, and let $(\mathcal{Q}, \|\cdot\|_{\cal Q})$ be a $q$-Banach polynomial ideal, $0 < q \leq 1$. Then, for every $m \geq 2$ odd, the subset of $\mathcal{Q}(^mE; V)$ of non-injective polynomials is pointwise spaceable. 
\end{corollary}

\begin{proof} By Example \ref{is9t} we know that $\mathcal{Q}(^mE;V)$ is a $(E,V)$-stable function space. Taking $0 \neq x^* \in E^*$ and $0 \neq b \in V$, $(x^*)^m\otimes b$ is a non-null non-injective $m$-homogeneous polynomial which belongs to $\mathcal{Q}(^mE;V)$ because $\cal Q$ is a polynomial ideal. The result follows from Theorem \ref{7}.
\end{proof}

\begin{example}\rm Let $\mathbb{K} = \mathbb{R}$, let $E$ and $F$ be a Banach spaces and let $(\mathcal{Q}, \|\cdot\|_{\cal Q})$ be a $q$-Banach polynomial ideal, $0 < q \leq 1$. For any $1 \leq p \leq \infty$, every $V \in V_{p,F}$ and every $m \geq 2$ odd, from Example  \ref{6ymz} and the corollary above it follows that the subset of ${\cal Q}(^mE;V)$ of non-injective polynomials is pointwise spaceable. In particular, the sets of non-injective compact $m$-homogeneous polynomials (with the usual polynomial norm) and of non-injective nuclear $m$-homogeneous polynomials (with the nuclear norm) from $E$ to $V$, which are, in general, much smaller than the set of non-injective continuous $m$-homogeneous polynomials, are pointwise spaceable.
\end{example}

Injectivity of Lipschitz functions was thoroughly investigated in \cite{malasia}. The presence of non-null constant functions makes the pointwise spaceability of the set of non-injective bounded Lipschitz functions very simple: 

\begin{proposition}\label{pac8} Let $(M,d)$ be a metric space and let $E$ be a Banach space. Then the set of non-injective Lipschitz functions $f \colon M \longrightarrow$ E is pointwise {\rm dim}$E$-spaceable in ${\rm Lip}(M,E)$.
\end{proposition}

\begin{proof} For every $x \in E$, denote by $i_x \colon M \longrightarrow E$ the constant function $i_x(w) = x$ for every $w \in M$. It is clear that $i_x \in {\rm Lip}(M,E)$ and that the operator $x \in E \mapsto i_x \in {\rm Lip}(M,E)$ is an isometric isomorphism onto its range. Let $f \in {\rm Lip}(M,E)$ be a non-injective function. If $f$ is constant, put $X = \{i_x : x \in E\}$; and if $f$ is non-constant, put $X = {\rm span}\{f\} \oplus \{i_x : x \in E\}$. In both cases, $X$ is a closed dim$E$-dimensional subspace of ${\rm Lip}(M,E)$ containing $f$ and contained in the set of non-injective functions.
\end{proof}

It is clear that, when dealing with spaces of Lipschitz functions, it is not interesting to be confined to constant functions. In our opinion, when the mother function $f$ is non-constant, the solution above is somewhat disappointing, in the sense that 
  the resulting space $X$ is formed by functions that are constant modulo the mother function $f$. More precisely,  every function in $X$ differs from a multiple of $f$ by a constant function. The space obtained in the next application of our  main result (and its proof) contains functions of the form $\lambda f + g$ for every $g$ belonging to an infinite dimensional space formed, up to the null function, by non-constant functions. 

\begin{proposition}\label{an7e} Let $(M,d)$ be a metric space, let $W$ be a linear space and let $V$ be a $W$-spreading space. For every non-injective non-constant Lipschitz function $f \colon M \longrightarrow V$, there exists a closed infinite dimensional subspace $X$ of ${\rm Lip}(M,V) $ such that: \\
{\rm (i)} $f \in X$ and every function in $X$ is non-injective. \\
{\rm (ii)} There exists an infinite dimensional subspace $Z$ of $X$ formed, up to the null function, by non-constant functions, 
 such that $ Z \cap {\rm span}\{f\}=\{0\}$ and $\lambda f + g \in X$, hence  $\lambda f + g$ is non-injective, for every $\lambda \in \mathbb{K}$ and every $g \in Z$.

\begin{proof} We use the notation of the proof of Theorem \ref{7}.\\
(i) By Example  \ref{l0mz} we know that ${\rm Lip}(M,V)$ is a $(M,V)$-stable function space. Since $0 \neq f \in {\rm Lip}(M,V)$ is non-injective, from the proof of Theorem \ref{7} we know that $X := \overline{\psi(\ell_1)}$ is a closed infinite dimensional subspace of ${\rm Lip}(M,V)$ containing $f$ and contained in the set of non-injective functions.\\
{\rm (ii)} Also from the proof of Theorem \ref{7}, we know that $Z := {\rm span} \{f_k : k \in \mathbb{N}\}$ is an infinite dimensional subspace of $X$ such that $ Z \cap {\rm span}\{f\}=\{0\}$. Given $\lambda \in \mathbb{K}$ and $g \in Z$, there are 
$a_1, \ldots , a_k \in \mathbb{K}$  such that
$$\lambda f + g = \lambda f +  a_1f_1+ \cdots + a_kf_k= \psi ((\lambda, a_1, \dots , a_k, 0, 0, \dots )) \in \psi (\ell _1) \subseteq X.$$
All that is left to be proved is that every non-null function of $Z$ is non-constant. Given $0 \neq g \in Z$, we can write $g = a_1f_1 + \cdots + a_k f_k$ where $a_j \neq 0$ for some $j \in \{1, \ldots, k\}$. Using that $f$ is non-constant, let $w_1, w_2 \in M$ be such that $f(w_1)\neq f(w_2)$. In this case there is $i \in \mathbb{N}$ so that $(f(w_1))_i \neq (f(w_2))_i$. For every $m \in \mathbb{N}$, from (\ref{ghn4}) we get
\begin{equation} \label{eq1}
(f_m(w_1))_{n_i^{(m)}}=(f(w_1))_i \neq (f(w_2))_i = (f_m(w_2))_{n_i^{(m)}},
\end{equation}
hence $f_m(w_1) \neq f_m(w_2)$. This proves that each $f_m$ is non-constant, in particular, $f_j$ is non-constant. 
Taking the $n_i^{(j)}$-coordinates of $g(w_1)$ and $g(w_2)$, we get, for $r = 1,2$,
\begin{align}\label{eq2}
(g(w_r))_{n_i^{(j)}}&= a_1(f_1(w_r))_{n_i^{(j)}}+ \cdots + a_j(f_j(w_r))_{n_i^{(j)}} + \cdots + a_k(f_k(w_r))_{n_i^{(j)}},
\end{align}
As $n_i^{(j)} \notin \bigcup\limits_{\stackrel{\ell =1}{\ell \neq j}} ^k \mathbb{N_\ell} $, by (\ref{ghn4}) we have $(f_{\ell}(w_1))_{n_i^{(j)}}=f_{\ell}(w_2))_{n_i^{(j)}}=0$  for $\ell =1,\dots , j-1, j+1, \dots , k. $ Combining (\ref{eq1}), (\ref{eq2}) and using that $a_j \neq 0$, we conclude that
$$(g(w_1))_{n_i^{(j)}}=a _j (f_j{(w_1)})_{n_i^{(j)}}\neq a _j (f_j{(w_2)})_{n_i^{(j)}}=(g(w_2))_{n_i^{(j)}}, $$
hence $g(w_1)\neq g(w_2)$. This proves that $g$ is non-constant.
\end{proof}
\end{proposition}

The proof of Proposition \ref{pac8} does not apply to ${\rm Lip}_0$-spaces because these spaces do not contain non-null constant functions. But our main theorem applies: the following corollary is a combination of Theorem \ref{7} and Example \ref{km6u}.

\begin{corollary} Let $(M,d)$ be a pointed metric space, let $W$ be a linear space and let $V$ be a $W$-spreading space. Then the set of non-injective Lipschitz functions is either $\{0\}$ or pointwise spaceable in ${\rm Lip}_0(M,V)$.
\end{corollary}

In the same way we did before, one obtains concrete applications of Proposition \ref{an7e} and of the corollary above using the spaces from Example \ref{6ymz}.

%

\bigskip
\noindent Mikaela Aires~~~~~~~~~~~~~~~~~~~~~~~~~~~~~~~~~~~~~~~~~~~~~~Geraldo Botelho~\\
Instituto de Matem\'atica e Estat\'istica~~~~~~~~~~~~~~~Faculdade de Matem\'atica\\
Universidade de S\~ao Paulo~~~~~~~~~~~~~~~~~~~~~~~~~~~~~\hspace*{0,1em}Universidade Federal de Uberl\^andia\\
05.508-090 -- S\~ao Paulo -- Brazil~~~~~~~~~~~~~~~~~~~~~~\,\hspace*{0,1em}38.400-902 -- Uberl\^andia -- Brazil\\
e-mail: mikaela\_aires@ime.usp.br~~~~~~~~~~~~~~~~~~~~~\,e-mail: botelho@ufu.br
\bigskip


%

%
%
%

\end{document}